\documentclass[12pt,twoside]{amsart}
\usepackage{amsmath}
\usepackage{amsthm}
\usepackage{amsfonts}
\usepackage{amssymb}
\usepackage{latexsym}
\usepackage{mathrsfs}
\usepackage{amsmath}
\usepackage{amsthm}
\usepackage{amsfonts}
\usepackage{amssymb}
\usepackage{latexsym}
\usepackage{geometry}
\usepackage{dsfont}
\usepackage[dvips]{graphicx}
\usepackage[colorlinks=true,linkcolor=red,citecolor=blue]{hyperref}
\usepackage{color}
\usepackage[all]{xy}

\date{}
\allowdisplaybreaks[4] \footskip=15pt
\renewcommand{\uppercasenonmath}[1]{}

\theoremstyle{plain}
\newtheorem{theorem}{Theorem}[section]
\newtheorem{proposition}[theorem]{Proposition}
\newtheorem{lemma}[theorem]{Lemma}
\newtheorem{corollary}[theorem]{Corollary}
\theoremstyle{definition}
\newtheorem{example}[theorem]{Example}
\newtheorem{definition}[theorem]{Definition}

\theoremstyle{definition}

\newcommand{\Q}{\mathcal{Q}}

\newcommand{\Pp}{\mathcal{P}}

\def\bc{\begin{center}}
\def\ec{\end{center}}

\def\Max{{\rm Max}}

\def\Ext{{\rm Ext}}
\def\Tor{{\rm Tor}}
\def\tor{{\rm tor}}

\def\fd{{\rm fd}}

\def\Hom{{\rm Hom}}
\def\GV{{\rm GV}}
\def\Max{{\rm Max}}

\def\fd{{\rm fd}}

\def\GV{{\rm GV}}
\def\tor{{\rm tor_{\rm GV}}}
\def\Hom{{\rm Hom}}
\def\Ext{{\rm Ext}}
\def\Tor{{\rm Tor}}

\def\fd{{\rm fd}}

\def\GV{{\rm GV}}
\def\tor{{\rm tor_{\rm GV}}}
\def\Hom{{\rm Hom}}
\def\Ext{{\rm Ext}}
\def\Tor{{\rm Tor}}

\def\Max{{\rm Max}}
\def\DW{{\rm DW}}

\def\PvMR{{\rm PvMR}}

\def\DQ{{\rm DQ}}
\def\WQ{{\rm WQ}}

\def\p{{\frak p}}
\def\m{{\frak m}}

\def\a{{\bf a}}
\def\A{{\bf A}}

\begin{document}
\begin{center}
{\large  \bf A homological characterization of $Q_0$-Pr\"{u}fer $v$-multiplication rings}

\vspace{0.5cm}   

Xiaolei Zhang\\
School of Mathematics and Statistics, Shandong University of Technology\\
Zibo 255049, China\\
E-mail: zxlrghj@163.com\\
\end{center}

\bigskip
\centerline { \bf  Abstract}
\bigskip
\leftskip10truemm \rightskip10truemm \noindent
Let $R$ be a commutative ring.  An $R$-module $M$ is called a  semi-regular $w$-flat module if $\Tor_1^R(R/I,M)$ is $\GV$-torsion for any finitely generated semi-regular ideal $I$. In this article, we show that the class of semi-regular $w$-flat modules is a covering class. Utilizing these notions, we give some homological characterizations of $\WQ$-rings and $Q_0$-\PvMR s.
\\
\vbox to 0.3cm{}

\leftskip0truemm \rightskip0truemm
\bigskip

\section{Introduction}

Throughout this paper, we always assume $R$ is a commutative ring with identity  and $T(R)$ is the total quotient ring of $R$. Following from \cite{fk16}, an ideal $I$ of $R$ is said to be \emph{dense} if $(0:_RI):=\{r\in R|Ir=0\}=0$ and  be \emph{semi-regular} if it contains a finitely generated dense sub-ideal.  Denote by $\Q$ the set of all finitely generated semi-regular ideals of $R$. Following from \cite{wzcc20} that a ring $R$ is called a \emph{$\DQ$-ring} if $\Q=\{R\}$. If $R$ is an integral domain, the quotient field $K$ is a very important $R$-module to study integral domains. However, the total quotient ring $T(R)$ is not always convenient to study commutative rings $R$ with zero-divisors. For example, the polynomial ring $R[x]$ is not always integrally closed in $T(R[x])$ when $R$ is integrally closed in $T(R)$ (see \cite{BCM79}). It is well-known that a finitely generated ideal $I=\langle a_0,a_1,\cdots,a_n\rangle$ is semi-regular if and only if the polynomial $f(x)=a_0+a_1x+\cdots+a_nx^n$ is a regular element in $R[x]$ (see \cite[Exercise 6.5]{fk16} for example). So, to study the integrally closeness of $R[x]$, Lucas \cite{L89} introduced the ring of finite fractions of $R$:
$$Q_0(R):=\{\alpha\in T(R[x])\mid\ \mbox{there exists}\ I\in \Q\ \mbox{such that } I\alpha\subseteq R\},$$
and showed that a reduced ring $R$ is integrally closed in $Q_0(R)$ if and only if $R[x]$ is integrally closed in $T(R[x])$.
Note that for any commutative ring $R$, we have $R\subseteq T(R)\subseteq Q_0(R)$. Recently, the authors \cite{z21,zdxq21} gave several homological characterizations of  total quotients rings (i.e. $R=T(R)$) and $\DQ$-rings utilizing certain generalized flat modules. There is a natural question to  characterize commutative rings with $R=Q_0(R)$ (called $\WQ$-rings from the star operation point of view). Actually, we show that  $\WQ$-rings are exactly those rings whose modules are all semi-regular $w$-flat (see Theorem \ref{201}).

Pr\"{u}fer domains are well-known domains and have been studied by many algebraists.  In order to generalize  Pr\"{u}fer domains  to commutative rings with zero-divisors, Butts and Smith \cite{BS67}, in 1967, introduced the notion of \emph{Pr\"{u}fer rings} over which every finitely generated regular ideal is invertible. Later in 1985, Anderson et al. \cite{AA85} introduced the notion of \emph{strong Pr\"{u}fer rings} whose finitely generated semi-regular ideals are all $Q_0$-invertible.  Strong Pr\"{u}fer rings have many nice properties. For example, the small finitistic dimensions of strong Pr\"{u}fer rings are at most one (see \cite{wzq20}). To give a $w$-analogue of  Pr\"{u}fer rings, Huckaba and Papick \cite{HP80} and Matsuda \cite{M80} called a ring $R$ to be a \PvMR\ (short for \emph{Pr\"{u}fer $v$-multiplication ring}) provided that any finitely generated regular ideal is $t$-invertible. For generalizing strong Pr\"{u}fer rings, Lucas \cite{L05} said a ring $R$ to be a $Q_0$-\PvMR\ (short for \emph{$Q_0$-Pr\"{u}fer $v$-multiplication ring}) if any finitely generated semi-regular ideal is $t$-invertible, and then he considered the properties of polynomial rings $R[x]$ and Nagata rings  $R(x)$ and $w$-Nagata rings  $R\{x\}$ when $R$ is a $Q_0$-\PvMR. For studying $Q_0$-\PvMR s, Qiao and Wang \cite{QW16} introduced quasi-$Q_0$-\PvMR s and showed that a ring $R$ is a $Q_0$-\PvMR\ if and only if $R$ is a quasi-$Q_0$-\PvMR, and $R$ has Property $B$ in $Q_0(R)$, i.e., $(IQ_0(R))_w=Q_0(R)_w$ for any $I\in\Q$. Wang and Kim \cite{fk15} gave some module-theoretic properties of $Q_0$-\PvMR s. Actually, they showed that  $Q_0$-\PvMR s are $Q_0$-$w$-coherent and each finite type semi-regular ideal is $w$-projective. The authors  \cite{z21,zdxq21} also gave some homological characterizations of  strong Pr\"{u}fer rings and $\PvMR$s utilizing the generalized flat modules. In this paper, we obtain several module-theoretic and  homological  characterizations of $Q_0$-\PvMR s using $w$-projective modules, $w$-flat modules and  semi-regular $w$-flat modules (see Theorem \ref{203}).

As our work involves the $w$-operations,  we give some reviews. A finitely generated ideal   $J$  of $R$ is called a \emph{Glaz-Vasconcelos ideal} (\emph{$\GV$-ideal} for short) if the natural homomorphism $R\rightarrow \Hom_R(J,R)$ is an isomorphism, and the set of all $\GV$-ideals is denoted by $\GV(R)$. Let $M$ be an $R$-module. Define
\begin{center}
{\rm $\tor(M):=\{x\in M|Jx=0$, for some $J\in \GV(R) \}.$}
\end{center}
An $R$-module $M$ is called \emph{$\GV$-torsion} (resp., \emph{$\GV$-torsion-free}) if $\tor(M)=M$ (resp., $\tor(M)=0$). A $\GV$-torsion-free module $M$ is called a \emph{$w$-module} if $\Ext_R^1(R/J,M)=0$ for any $J\in \GV(R)$, and the \emph{$w$-envelope} of $M$ is given by
\begin{center}
{\rm $M_w:=\{x\in E(M)|Jx\subseteq M$, for some $J\in \GV(R) \},$}
\end{center}
where $E(M)$ is the injective envelope of $M$. A fractional ideal $I$ is said to be $w$-invertible if $(II^{-1})_w=R$. A \emph{$\DW$ ring} $R$ is a ring over which every module is a $w$-module, equivalently the only $\GV$-ideal of $R$  is $R$. A \emph{maximal $w$-ideal} is an ideal of $R$ which is maximal among all $w$-submodules of $R$. The set of all maximal $w$-ideals is denoted by $w$-$\Max(R)$, and any maximal $w$-ideal is prime.

An $R$-homomorphism $f:M\rightarrow N$ is said to be a \emph{$w$-monomorphism} (resp., \emph{$w$-epimorphism}, \emph{$w$-isomorphism}) if for any $ \m \in w$-$\Max(R)$, $f_{\m}:M_{\m}\rightarrow N_{\m}$ is a monomorphism (resp., an epimorphism, an isomorphism).  A sequence $A\rightarrow B\rightarrow C$ is said to be \emph{$w$-exact} if for any $ \m\in w$-$\Max(R)$, $A_{\m}\rightarrow B_{\m}\rightarrow C_{\m}$ is exact. A class $\mathcal{C}$ of $R$-modules is said to be closed under $w$-isomorphisms provided that for any $w$-isomorphism $f:M\rightarrow N$, if one of the modules $M$ and $N$ is in $\mathcal{C}$, so is the other. An $R$-module $M$ is said to be of \emph{finite type} if there exist a finitely generated free module $F$ and a $w$-epimorphism $g: F\rightarrow M$. Following from \cite{fk15}, an $R$-module $M$ is said to be \emph{$w$-flat} if for any $w$-monomorphism $f: A\rightarrow B$, the induced sequence $f\otimes_R 1:A\otimes_R M\rightarrow B\otimes_R M$ is also a $w$-monomorphism. The classes of finite type modules and $w$-flat modules are all closed under $w$-isomorphisms, see \cite[Corollary 6.7.4]{fk16}.

\section{Semi-regular $w$-flat modules}

Recall from \cite{zdxq21}, an $R$-module $M$ is said to be a semi-regular flat module if, for any finitely generated semi-regular ideal $I$ (i.e. $I\in\Q$), we have $\Tor_1^R(R/I,M)=0$.  Obviously, every flat module is semi-regular flat. We denote by $\mathcal{F}_{sr}$ the class of all semi-regular flat modules.  Then the class $\mathcal{F}_{sr}$ of all semi-regular flat modules is closed under direct limits, pure submodules and pure quotients \cite[Lemma 2.4]{zdxq21}. Hence $\mathcal{F}_{sr}$  is a covering class (see \cite[Theorem 2.6]{zdxq21}). Now, we give a $w$-analogue of semi-regular flat modules.

\begin{definition}\label{sr-flat }
An $R$-module $M$ is said to be a \emph{semi-regular $w$-flat module} if $\Tor_1^R(R/I,M)$ is $\GV$-torsion for any $I\in\Q$. The class of all semi-regular $w$-flat modules is denoted by $w$-$\mathcal{F}_{sr}$.
\end{definition}

Obviously, semi-regular flat modules and  $w$-flat modules are all semi-regular $w$-flat.
Following from \cite{zdxq21} that an $R$-module $M$ is said to be a semi-regular coflat module if for any  $I\in\Q$, we have $\Ext^1_R(R/I,M)=0$.

\begin{lemma}\label{101}
 Let $M$ be an $R$-module. Then the following statements are equivalent:
\begin{enumerate}
 \item $M$ is a semi-regular $w$-flat module;
 \item for any  $I\in\Q$, the natural homomorphism $I\otimes M\rightarrow R\otimes M$ is a $w$-monomorphism;
 \item for any  $I\in\Q$, the natural homomorphism $\sigma_I: I\otimes M\rightarrow IM$ is a $w$-isomorphism;
 \item for any injective $w$-module $E$, $\Hom_R(M,E)$ is a semi-regular coflat module;
\end{enumerate}
\end{lemma}
\begin{proof} $(1)\Leftrightarrow (2)$: Let $I$ be a finitely generated semi-regular ideal. Then we have a long exact sequence:  $$0\rightarrow \Tor_1^R(R/I,M)\rightarrow I\otimes_R M\rightarrow R\otimes_R M\rightarrow R/I\otimes_R M\rightarrow 0.$$ Consequently, $\Tor_1^R(R/I,M)$ is $\GV$-torsion if and only if
$I\otimes_R M\rightarrow R\rightarrow R\otimes_R M$ is a $w$-monomorphism.

$(2)\Rightarrow (3)$: Let $I$ be a finitely generated semi-regular ideal. Then we have the following commutative diagram:
$$\xymatrix{
 0 \ar[r]^{} &I\otimes_R M \ar[d]_{\sigma_I}\ar[r]^{} & R\otimes_RM \ar[d]_{\cong}\\
 0 \ar[r]^{} & IM \ar[r]^{} &M.
}$$
Then $\sigma_I$ is a $w$-monomorphism. Since the multiplicative map $\sigma_I$ is an epimorphism, $\sigma_I$ is a $w$-isomorphism.

$(3)\Rightarrow (1)$: Let $I$ be a finitely generated semi-regular ideal. Then we have a long exact sequence:
$$\xymatrix{
 0\ar[r]^{} & \Tor_1^R(R/I,M) \ar[r]^{} &I M \ar[r]^{f} & M.
}$$
Since $f$ is a  natural embedding map, we have $\Tor_1^R(R/I,M)$ is $\GV$-torsion.

$(1)\Rightarrow (4)$: Let $I$ be a finitely generated semi-regular ideal and $E$  an injective $w$-module. Then $\Ext_R^1(R/I,\Hom_R(M,E))\cong\Hom_R(\Tor_1^R(R/I,M),E)$. Since $M$ is a semi-regular $w$-flat module, then $\Tor_1^R(R/I,M)$ is $\GV$-torsion. Since $E$ is a $w$-module, we have $\Hom_R(\Tor_1^R(R/I,M),E)=0$. Thus $\Ext_R^1(R/I,\Hom_R(M,E))=0$, So $\Hom_R(M,E)$ is a semi-regular coflat module.

$(4)\Rightarrow (1)$: Let $I$ be a finitely generated semi-regular ideal and $E$  an injective $w$-module. Since $\Hom_R(M,E)$ is a semi-regular coflat module and $$\Ext_R^1(R/I,\Hom_R(M,E))\cong\Hom_R(\Tor_1^R(R/I,M),E),$$ we have $\Hom_R(\Tor_1^R(R/I,M),E)=0$. By \cite[Corollary 3.11]{zw21-1}, $\Tor_1^R(R/I,M)$ is $\GV$-torsion. So $M$ is a semi-regular $w$-flat module.
\end{proof}

\begin{corollary}\label{w-closed}
Let $R$ be a ring. The class of semi-regular $w$-flat $R$-modules is closed under $w$-isomorphisms.
\end{corollary}
\begin{proof} Let $f:M\rightarrow N$ be a $w$-isomorphism and $I$  a finitely generated semi-regular  ideal.  There exist two exact sequences $0\rightarrow T_1\rightarrow M\rightarrow L\rightarrow 0$ and $0\rightarrow L\rightarrow N\rightarrow T_2\rightarrow 0$ with $T_1$ and $T_2$ $\GV$-torsion.
Consider the induced two long exact sequences $\Tor^R_1(R/I,T_1)\rightarrow \Tor^R_1 (R/I,M)\rightarrow \Tor^R_1 (R/I,L)\rightarrow R/I\otimes T_1$ and $\Tor^R_2 (R/I,T_2)\rightarrow \Tor^R_1(R/I,L)\rightarrow \Tor^R_1(R/I,N)\rightarrow \Tor^R_1(R/I,T_2)$. By \cite[Theorem 6.7.2]{fk16},   $M$ is semi-regular $w$-flat if and only if $N$ is semi-regular $w$-flat.
\end{proof}

\begin{proposition}\label{102}
Let $R$ be a ring. Then $R$ is a $\DW$-ring if and only if any semi-regular $w$-flat module is semi-regular flat.
\end{proposition}
\begin{proof}
Obviously, if $R$ is a $\DW$-ring, then  every  semi-regular $w$-flat module is semi-regular flat.  On the other hand, let $J$ be a $\GV$-ideal of $R$, then $R/J$  is $\GV$-torsion and hence a semi-regular $w$-flat module by Corollary \ref{w-closed}. So $R/J$ is a semi-regular flat module. Note the $\GV$-ideal  $J$ is finitely generated and semi-regular, so $\Tor_1^R(R/J,R/J)\cong J/J^2=0$ by \cite[Exercise 3.20]{fk16}. It follows that $J$ is a finitely generated idempotent ideal of $R$, and thus  $J$ is projective  by \cite[Proposition 1.10]{FS01}.  Hence $J=J_w=R$. Consequently,  $R$ is a $\DW$-ring.
\end{proof}

We say a class  $ \mathcal{F}$ of $R$-modules  is  precovering  provided that  for any $R$-module $M$, there is a homomorphism $f: F\rightarrow M$ with  $F\in \mathcal{F}$ such that $\Hom_R(F',F)\rightarrow \Hom_R(F',M)$ is an epimorphism for any $F'\in \mathcal{F}$. If, moreover, any homomorphism $h$ such that $f=f\circ h$  is an isomorphism,  $\mathcal{F}$ is said to be covering. It is well-known that the class of flat modules is a covering class (see \cite[Theorem 3]{BBE01}). It was also proved in \cite[Theorem 3.5]{z19}  that  the class of $w$-flat modules is a covering class. For the class of semi-regular flat modules, we have the following similar result.
\begin{proposition}\label{102}
Let $R$ be a ring. Then the class $w$-$\mathcal{F}_{sr}$ of all semi-regular flat modules is closed under direct limits, pure submodules and pure quotients. Consequently, $w$-$\mathcal{F}_{sr}$ is a covering class.
\end{proposition}
\begin{proof}
For the  direct limits, suppose $\{M_i\}_{i\in\Gamma}$ is a direct system consisting of semi-regular $w$-flat modules. Then, for any finitely generated semi-regular ideal $I$, we have $\Tor_1^R(R/I,\lim\limits_{\longrightarrow } M_i)=\lim\limits_{\longrightarrow}\Tor_1^R(R/I,M_i)$ is $\GV$-torsion. So $\lim\limits_{\longrightarrow }M_i$ is a semi-regular $w$-flat module.

For pure submodules and pure quotients, let $I$ be a finitely generated semi-regular ideal. Suppose $0\rightarrow M\rightarrow N\rightarrow L\rightarrow 0$ is a pure exact sequence. We have the following commutative diagram with rows exact :
$$\xymatrix{
 0 \ar[r]^{} & M\otimes_R I \ar[d]_{f}\ar[r]^{} & N\otimes_R I \ar@{>->}[d]_{}\ar[r]^{} & L\otimes_R I \ar[d]_{g}\ar[r]^{} & 0\\
 0 \ar[r]^{} & M\otimes_R R \ar@{->>}[d]_{}\ar[r]^{} & N\otimes_R R \ar@{->>}[d]_{}\ar[r]^{} & L\otimes_R R \ar@{->>}[d]_{}\ar[r]^{} & 0\\
 0 \ar[r]^{} & M\otimes_R R/I \ar[r]^{} & N\otimes_R R/I \ar[r]^{} & L\otimes_R R/I \ar[r]^{} & 0\\}$$
By the generalized Five Lemma (see \cite[Lemma 6.3.6]{fk16}), the natural homomorphism $f: M\otimes_R I \rightarrow M\otimes_R R$ and $g: L\otimes_R I \rightarrow L\otimes_R R$ are all $w$-monomorphisms. Consequently, $M$ and $L$ are all semi-regular $w$-flat. Consequently, $w$-$\mathcal{F}_{sr}$ is a covering class  by \cite[Theorem 3.4]{HJ08}.
\end{proof}

\section{Rings characterized by Semi-regular $w$-flat modules}

Recently, the authors in \cite{ZDC20} introduced the notion of $q$-operation on a commutative ring $R$. We give some reviews here. An $R$-module $M$ is said to be $\Q$-torison-free if $Im=0$ with $I\in\Q$ and $m\in M$ can deduce $m=0$. Let $M$ be a $\Q$-torison-free $R$-module. The Lucas envelope
\begin{center}
$M_q=\{x\in E(M)\mid\ \mbox{there exists }\ I\in\Q\ \mbox{such that } Ix\subseteq M\}$
\end{center}
where $E(M)$ is the injective envelope of $M$. An $\Q$-torison-free $R$-module $M$ is said to be a Lucas module provided that $M_q=M$. By \cite[Proposition 2.2]{fkxs20}, a ring is a $\DQ$-ring if and only if every $R$-module is Lucas module. Since any $\GV$-ideal is finitely generated semi-regular, we have Lucas modules are all $w$-modules. However, $R$ itself is not always a Lucas module.  It was proved in \cite[Proposition 3.8]{wzcc20} that a ring $R$ is a Lucas module if and only if the $q$- and $w$-operations on $R$ coincide, if and only if every finitely generated semi-regular ideal is a $\GV$-ideal, if and only if $Q_0(R)=R$. For convenience, we say a ring $R$ is a $\WQ$-ring if every finitely generated semi-regular ideal is a $\GV$-ideal. Obviously, a ring $R$ is a $\DQ$-ring if and only  if it is both a $\DW$-ring and a $\WQ$-ring.

\begin{lemma}\label{fangmi}
Let $I=\langle a_1, a_2,\cdots,a_n\rangle$ be a finitely generated ideal of $R$. Suppose $m$ is  a positive integer and $K=\langle a^m_1, a^m_2,\cdots,a^m_n\rangle$. Then $I^{mn}\subseteq K$.
\end{lemma}
\begin{proof}
Note that $I^{mn}$ is generated by $\{\prod\limits_{i=1}^na_i^{k_i}\mid  \sum\limits_{i=1}^nk_i=mn\}$. By the pigeonhole principle, there exits some $k_i$ such that $k_i\geq m$. So each $\prod\limits_{i=1}^na_i^{k_i}\in K$, and thus $I^{mn}\subseteq K$.
\end{proof}

\begin{theorem}\label{201}
Let $R$ be a ring. Then the following statements are equivalent:
\begin{enumerate}
 \item $R$ is a $\WQ$-ring;
 \item every $R$-module is semi-regular $w$-flat;
 \item for every finitely generated semi-regular ideal $I$, $R/I$ is a $w$-flat module;
 \item  $I\subseteq (I^2)_w$ for any finitely generated semi-regular ideal $I$ of $R$;
  \item every $w$-module is a Lucas module;
  \item $Q_0(R)=R$.
\end{enumerate}
\end{theorem}
\begin{proof}
$(1)\Rightarrow (2)$: Let $I$ be a finitely generated semi-regular ideal of $R$ and $M$ an $R$-module. Then $I$ is a $\GV$-ideal of $R$. So $\Tor_1^R(R/I,M)$ is $\GV$-torsion. Hence $M$ is semi-regular $w$-flat.

$(1)\Rightarrow (5)$: Trivial.

$(2)\Rightarrow (3)$: Let $I$ be a finitely generated semi-regular ideal of $R$ and $K$ a finitely generated ideal of $R$. Then $R/K$ is a semi-regular $w$-flat module. So $\Tor_1^R(R/K,R/I)$ is $\GV$-torsion. Hence $R/I$ is a $w$-flat module by \cite[Theorem 6.7.3]{fk16}.

$(3)\Rightarrow (4)$: Let $I$ be a finitely generated semi-regular ideal of $R$. Then  $\Tor_1^R(R/I,R/I)$ is $\GV$-torsion since $R/I$ is a $w$-flat module by $(4)$. That is, $I/I^2$ is $\GV$-torsion, and thus $I\subseteq (I)_w=(I^2)_w$.

$(4)\Rightarrow (1)$: Let  $I=\langle a_1,...,a_n\rangle$ is finitely generated semi-regular ideal. There exists a $\GV$-ideal $J$ such that $JI\subseteq I^2$.  We claim that $I$ is also a $\GV$-ideal. Indeed,                                                               suppose $J$ is generated by $\{j_1,\cdots ,j_m\}$. For each $k=1,\cdots,m$, we have $j_ka_i=\sum\limits_{j=1}^nr_{ij}a_j$ for some suitable $r_{ij}\in I$. The column vector $\a\in R^n$ whose $i$-th coordinate is $a_i$, and the matrix $\A=\|j_k\delta_{ij}-r_{ij}\|$, where $\delta_{ij}$ is the Kronecker symbol, satisfy $\A\a=0$. Hence $\det(\A)\a=0$. Since  $I$ is semi-regular, we have $\det(\A)=0$. So $j^n_k+j^{n-1}_kr_1+\cdots+r_n=0$ for some $r_i\in I$. Thus $j^n_k\in I$ for each $k=1,\cdots,m$. By Lemma \ref{fangmi}, we have $J^{nn}\subseteq \langle j^n_k\mid k=1,\cdots,m\rangle \subseteq I$. Since $J^{nn}$ is a $\GV$-ideal, the finitely generated semi-regular ideal $I$ is also a $\GV$-ideal (see \cite[Proposition 6.1.9]{fk16}).

$(5)\Rightarrow (1)$: Since $R$ is a $w$-module, then it is a Lucas module. So $R$ is a $\WQ$-ring by \cite[Proposition 3.8]{wzcc20}.

$(1)\Leftrightarrow (6)$: See  \cite[Proposition 3.8]{wzcc20}.
\end{proof}

It was proved in \cite[Theorem 3.1]{zdxq21} that a ring $R$ is a $\DQ$-ring (i.e. the only finitely generated semi-regular ideal of $R$ is $R$ itself)  if and only if every $R$-module is  semi-regular flat.
\begin{example}\cite[Example 12]{L93}
Let $D=L[X^2,X^3,Y]$, $\Pp=Spec(D)-\{\langle X^2,X^3,Y\rangle\}$, $B=\bigoplus\limits_{\p\in \Pp}K(R/\p)$  and $R=D(+)B$ where $L$ is a field and $K(R/\p)$ is the quotient field of $R/\p$. Since $Q_0(R)=R$, $R$ is a $\WQ$-ring by \cite[Proposition 3.8]{wzcc20}. Since every finitely generated $R$-ideal of the form $J(+)B$ with $\sqrt{J}=\langle X^2,X^3,Y\rangle$ is a $\GV$-ideal, $R$ is not a $\DW$-ring. Hence $R$ is not a $\DQ$-ring by \cite[Proposition 2.2]{fkxs20}.
\end{example}

Recall from  Lucas \cite{L05} that an ideal $I$ of $R$ is said to be $t$-invertible if there
is an $R$-submodule $J$ of $Q_0(R)$ such that $(IJ)_t = R$, and $R$ is called a $Q_0$-$\PvMR$  if every finitely generated semi-regular ideal of $R$ is $t$-invertible. From \cite[Proposition 4.17]{fk15}, a semi-regular ideal is $t$-invertible if and only if it is  $w$-invertible. So a ring $R$ is a $Q_0$-$\PvMR$ if and only if every finitely generated semi-regular ideal is $w$-invertible. Recall from \cite{fk15} that an $R$-module $M$ is said to be a \emph{$w$-projective module} if $\Ext^1_R((M/\Tor_{\GV}(M))_w, N)$ is a $\GV$-torsion module for any torsion-free $w$-module $N$. Recall from \cite{fk15} that an $R$-module $M$ is said to be semi-regular if there are a positive integer $n$ and a chain of submodules of $M$:
$$0\subseteq M_1\subseteq  M_2\subseteq \cdots \subseteq  M_{n-1}\subseteq  M_n=M$$
such that every factor module $M_i/M_{i-1}$ is $w$-isomorphic to a semi-regular ideal of $R$.

\begin{theorem}\label{203}
 Let $R$ be a ring. Then the following statements are equivalent:
\begin{enumerate}
 \item $R$ is a  $Q_0$-$\PvMR$;
 \item any submodule of a semi-regular $w$-flat $R$-module  is semi-regular $w$-flat;
 \item any submodule of a $w$-flat $R$-module  is semi-regular $w$-flat;
 \item any ideal of $R$ is semi-regular $w$-flat;
 \item any finitely generated $($resp., finite type$)$ ideal of $R$ is semi-regular $w$-flat;
 \item any finitely generated  $($resp., finite type$)$ semi-regular ideal  of $R$ is $w$-flat;
 \item any finitely generated  $($resp., finite type$)$ semi-regular ideal  of $R$ is $w$-projective;
  \item  any finitely generated  $($resp., finite type$)$ semi-regular $R$-module is $w$-projective.
\end{enumerate}
\end{theorem}
\begin{proof} Since the classes of semi-regular $w$-flat modules, $w$-flat modules, $w$-projective modules and $w$-invertible ideals are closed under $w$-isomorphism and every  finite type ideal is isomorphic to a finitely generated  sub-ideal, we just need to consider the ``finitely generated'' cases in $(5), (6)$, $(7)$ and $(8)$.

$(2)\Rightarrow (3)\Rightarrow (4)\Rightarrow (5)$, $(7)\Rightarrow (6)$ and  $(8)\Rightarrow (5)$: Trivial.

$(5)\Leftrightarrow (6)$: Let $I$ be a finitely generated semi-regular ideal of $R$ and $J$ a finitely generated ideal  of $R$. Then we have  $\Tor_1^R(R/J, I)\cong\Tor_2^R(R/I, R/J)\cong\Tor_1^R(R/I, J)$.
Consequently, $J$ is semi-regular $w$-flat if and only if $I$ is $w$-flat.

$(6)\Rightarrow (1)$: Let $I$ be a finitely generated semi-regular ideal of $R$ and $\m$  a maximal $w$-ideal of  $R$. Then $I_{\m}$ is finitely generated flat $R_{\m}$-ideal. By \cite[Lemma 4.2.1]{g} and \cite[Theorem 2.5]{M89}, we have $I_{\m}$ is a free $R_{\m}$-ideal. So the rank of  $I_{\m}$ is at most $1$. Hence $I$ is $w$-invertible by \cite[Theorem 4.13]{fk15}.

$(1)\Rightarrow (6)$: Let $I$ be a finitely generated semi-regular ideal of $R$ and $\m$  a maximal $w$-ideal of  $R$.  Then $I_{\m}$ is a principal $R_{\m}$-ideal  by \cite[Theorem 4.13]{fk15}. Suppose $I_{\m}=\langle \frac{x}{s}\rangle$. Then  $(0:_{R_{\m}}\frac{x}{s})=(0:_{R_{\m}}I_{\m})=(0:_RI)_{\m}=0$ by \cite[Exercise 1.72] {fk16}.  Thus $\frac{x}{s}$ is regular element. So $I_{\m}\cong R_{\m}$. Consequently, $I$ is a $w$-flat $R$-ideal.

$(6)\Rightarrow (2)$: Let $M$ be a semi-regular $w$-flat module and $N$  a submodule of $M$. Suppose $I$ is a finitely generated semi-regular ideal, then $I$ is a $w$-flat ideal. Thus $w$-$\fd_R(R/I)\leq 1$. Consider the exact sequence $$\Tor_2^R(R/I,M/N)\rightarrow \Tor_1^R(R/I,N)\rightarrow \Tor_1^R(R/I,M).$$ Since $\Tor_2^R(R/I,M/N)$ and $\Tor_1^R(R/I,M)$ are $\GV$-torsion, we have $\Tor_1^R(R/I,N)$ is $\GV$-torsion. So $N$ is a semi-regular $w$-flat module.

$(1)\Rightarrow (7)$: Let $I$ be a finitely generated semi-regular ideal of $R$. Then $I$ is $w$-invertible, and hence $w$-projective by \cite[Theorem 4.13]{fk15}.


$(1)\Rightarrow (8)$: See \cite[Theorem 4.23]{fk15}.
\end{proof}

The following examples show that regular $w$-flat ideals are not necessary semi-regular $w$-flat and semi-regular $w$-flat ideals are not necessary semi-regular flat.

\begin{example}\cite[Example 8.10]{L05}
Let $D = \mathbb{Z} + (Y,Z)\mathbb{Q}[[Y,Z]]$ and let $\Pp$ be
the set of height one primes of $D$. Let $R=B+B$ be the $A+B$ ring corresponding
to $D$ and $\Pp$. It was showed that $R$ is a $\PvMR$ but not a $Q_0$-$\PvMR$. Hence there exists a  regular $w$-flat ideal which is not  semi-regular $w$-flat by Theorem \ref{203} and \cite[Theorem 4.8]{z21}.
\end{example}

\begin{example}\cite[Example 8.11]{L05}
Let $E=D[Z]$ where $D$ is a Dedekind domain with a maximal ideal $N=\langle a, b\rangle$ for which no power of $N$ is principal. Let $\Pp$ be
the set of  primes of $E$ which contain neither $Z$ nor $NE$. Set $R = B + B$. It was showed that $R$ is a $Q_0$-$\PvMR$ but not a strong Pr\"{u}fer ring. Hence there exists a  semi-regular $w$-flat ideal which is not  semi-regular flat by Theorem \ref{203} and \cite[Theorem 3.4]{zdxq21}.
\end{example}

\end{document}